\documentclass[12pt]{amsart}
\usepackage{amsfonts,amssymb,amscd,amsmath,enumerate,color}
\def\NZQ{\mathbb}               

\def\ZZ{{\NZQ Z}}
\def\RR{{\NZQ R}}
\def\CC{{\NZQ C}}

\def\frk{\mathfrak}               

\def\Phi{{\frk N}}
\def\opn#1#2{\def#1{\operatorname{#2}}} 
\opn\chara{char} 
\opn\length{\ell} 
\opn\pd{pd} 
\opn\rk{rk}
\opn\projdim{proj\,dim} 
\opn\injdim{inj\,dim} 
\opn\rank{rank}
\opn\depth{depth} 
\opn\grade{grade} 
\opn\height{height}
\opn\embdim{emb\,dim} 
\opn\codim{codim}

\opn\Tr{Tr} 
\opn\bigrank{big\,rank}
\opn\superheight{superheight}
\opn\lcm{lcm}
\opn\trdeg{tr\,deg}
\opn\reg{reg} 
\opn\lreg{lreg} 
\opn\ini{in} 
\opn\lpd{lpd}
\opn\size{size}
\opn\mult{mult}
\opn\dist{dist}
\opn\cone{cone}
\opn\lex{lex}
\opn\rev{rev}
\opn\div{div} \opn\Div{Div} \opn\cl{cl} \opn\Cl{Cl}
\opn\Spec{Spec} \opn\Supp{Supp} \opn\supp{supp} \opn\Sing{Sing}
\opn\Ass{Ass} \opn\Min{Min}
\opn\Ann{Ann} \opn\Rad{Rad} \opn\Soc{Soc}
\opn\Syz{Syz} \opn\Im{Im} \opn\Ker{Ker} \opn\Coker{Coker}
\opn\Am{Am} \opn\Hom{Hom} \opn\Tor{Tor} \opn\Ext{Ext}
\opn\End{End} \opn\Aut{Aut} \opn\id{id} \opn\ini{in}

\opn\nat{nat}
\opn\pff{pf}
\opn\Pf{Pf} \opn\GL{GL} \opn\SL{SL} \opn\mod{mod} \opn\ord{ord}
\opn\Gin{Gin}
\opn\Hilb{Hilb}\opn\adeg{adeg}\opn\std{std}\opn\ip{infpt}
\opn\Pol{Pol}
\opn\sat{sat}
\opn\Var{Var}
\opn\Gen{Gen}

\opn\aff{aff} \opn\con{conv} \opn\relint{relint} \opn\st{st}
\opn\lk{lk} \opn\cn{cn} \opn\core{core} \opn\vol{vol}
\opn\link{link} \opn\star{star} \opn\Box{Box}
\opn\gr{gr}

\def\Hc{{\mathcal H}}

\def\Fc{{\mathcal F}}
\def\Pc{{\mathcal P}}
\def\Qc{{\mathcal Q}}

\def\pot#1#2{#1[\kern-0.28ex[#2]\kern-0.28ex]}

\opn\dirlim{\underrightarrow{\lim}}
\opn\inivlim{\underleftarrow{\lim}}
\def\Implies{\ifmmode\Longrightarrow \else
	\unskip${}\Longrightarrow{}$\ignorespaces\fi}
\def\implies{\ifmmode\Rightarrow \else
	\unskip${}\Rightarrow{}$\ignorespaces\fi}
\def\iff{\ifmmode\Longleftrightarrow \else
	\unskip${}\Longleftrightarrow{}$\ignorespaces\fi}

\let\:=\colon
\newtheorem{Theorem}{Theorem}[section]
\newtheorem{Lemma}[Theorem]{Lemma}
\newtheorem{Corollary}[Theorem]{Corollary}

\newtheorem{Remark}[Theorem]{Remark}

\newtheorem*{acknowledgement}{Acknowledgment}
\let\epsilon\varepsilon
\let\phi=\varphi
\let\kappa=\varkappa
\def\qed{\ifhmode\textqed\fi
	\ifmmode\ifinner\quad\qedsymbol\else\dispqed\fi\fi}
\def\textqed{\unskip\nobreak\penalty50
	\hskip2em\hbox{}\nobreak\hfil\qedsymbol
	\parfillskip=0pt \finalhyphendemerits=0}
\def\dispqed{\rlap{\qquad\qedsymbol}}

\opn\dis{dis}
\opn\height{height}
\opn\dist{dist}
\def\pnt{{\raise0.5mm\hbox{\large\bf.}}}

\opn\Lex{Lex}

\begin{document}
	\title{Gorenstein simplices and the associated finite abelian groups}
	\author{Akiyoshi Tsuchiya}
	\address{Akiyoshi Tsuchiya,
		Department of Pure and Applied Mathematics,
		Graduate School of Information Science and Technology,
		Osaka University, Suita, Osaka 565-0871, Japan}
	\email{a-tsuchiya@ist.osaka-u.ac.jp}
\begin{abstract}
	It is known that a lattice simplex of dimension $d$ corresponds a finite abelian subgroup of $(\mathbb{R}/\mathbb{Z})^{d+1}$.
	Conversely, given a finite abelian subgroup of $(\mathbb{R}/\mathbb{Z})^{d+1}$ such that the sum of all entries of each element is 
	an integer, we can obtain a lattice simplex of dimension $d$.
	In this paper, we discuss a characterization of Gorenstein simplices in terms of the associated finite abelian groups.
	In particular, we present complete characterizations of Gorenstein simplices whose normalized volume equals $p,p^2$ and $pq$,
	where $p$ and $q$ are prime numbers with $p \neq q$.
	Moreover, we compute the volume of the associated dual reflexive simplices of the Gorenstein simplices.  
\end{abstract}
\maketitle
	\section*{Introduction}
	A \textit{lattice polytope} is a convex polytope each of whose vertices has integer coordinates.	
		For a lattice simplex $\Delta \subset \RR^d$ of dimension $d$ whose vertices are $v_0,\ldots,v_d \in \ZZ^d$
		set 
		$$\Lambda_\Delta=\{(\lambda_0,\ldots,\lambda_d) \in (\RR/\ZZ)^{d+1} : \sum\limits_{i=0}^{d}\lambda_i(v_i,1) \in \ZZ^{d+1}   \}.$$
		The collection $\Lambda_\Delta$ forms a finite abelian group with addition defined as follows: 
		For $(\lambda_0,\ldots,\lambda_d) \in (\RR/\ZZ)^{d+1}$ and $(\lambda_0',\ldots,\lambda_d') \in (\RR/\ZZ)^{d+1}$,  $(\lambda_0,\ldots,\lambda_d)+(\lambda_0',\ldots,\lambda_d')=(\lambda_0+\lambda_0',\ldots,\lambda_d+\lambda_d') \in (\RR/\ZZ)^{d+1}$.
		Moreover, the order of $\Lambda_{\Delta}$ equals the \textit{normalized volume} of $\Delta$, i.e., 
		$d!$ times the usual euclidean volume of $\Delta$,
		which we denote by $\text{Vol}(\Delta)$.
		Let $\ZZ^{d \times d}$ be the set of $d \times d$ integral matrices.
Recall that a matrix $A \in \ZZ^{d \times d}$ is {\em unimodular} if $\det (A) = \pm 1$.
Given lattice polytopes $\Pc$ and $\Qc$ in $\RR^d$ of dimension $d$,
we say that $\Pc$ and $\Qc$ are {\em unimodularly equivalent}
if there exist a unimodular matrix $U \in \ZZ^{d \times d}$
and an integral vector $w$ such that $\Qc=f_U(\Pc)+w$,
where $f_U$ is the linear transformation in $\RR^d$ defined by $U$,
i.e., $f_U({\bf v}) = {\bf v} U$ for all ${\bf v} \in \RR^d$.
In \cite{BH}, it is shown that there is a bijection between unimodular equivalence classes of $d$-dimensional lattice simplices with a 
chosen ordering of their vertices and finite subgroups of $(\RR/\ZZ)^{d+1}$ such that the sum of all entries of each element is an integer.
In particular, two lattice simplices $\Delta$ and $\Delta'$ are unimodularly equivalent if and only if there exists an ordering of their vertices such that $\Lambda_\Delta=\Lambda_{\Delta'}$.

\smallskip
A lattice polytope $\Pc \subset \RR^d$ is called \textit{reflexive} if the origin of $\RR^d$ is the unique lattice point belonging to the interior of $\Pc$ and its dual polytope, i.e., 
	$$\Pc^\vee:=\{y \in \RR^d : \langle x,y\rangle \leq 1 \ \text{for all}\  x \in \Pc \},$$
	is also a lattice polytope, where $\langle x,y\rangle$ is the usual inner product of $\RR^d$.
We say that a lattice polytope $\Pc \subset \RR^d$ is \emph{Gorenstein of index} $r$ where $r\in \ZZ_{>0}$ if 
there exist a reflexive polytope $\Qc \subset \RR^d$ and a lattice point $w \in \ZZ^d$ such that $\Qc=r\Pc+w$ (\cite{DeNegriHibi}).
Equivalently, the semigroup algebra associated to the cone over $\Pc$ is a Gorenstein algebra.
We call $\Qc$ the \textit{associated dual reflexive polytope} of $\Pc$.
Gorenstein polytopes are of interest in combinatorial commutative algebra, mirror symmetry and tropical geometry
(for details we refer to \cite{Batyrev,JK}).
For a lattice polytope $\Pc \subset \RR^d$ of dimension $d$,
we can construct a new lattice polytope
$$\text{Pyr}(\Pc):=\text{conv}(\Pc \times \{0 \}, (0,\ldots,0,1)) \subset \RR^{d+1}$$
of dimension $d+1$.
This polytope $\text{Pyr}(\Pc)$ is called the \textit{lattice pyramid} over $\Pc$.
Then we have $\text{Vol}(\Pc)=\text{Vol}(\text{Pyr}(\Pc))$.
Moreover, it is known that $\Pc$ is Gorenstein of index $r$ if and only if $\text{Pyr}(\Pc)$ is Gorenstein of index $r+1$.
Hence, if we construct all Gorenstein polytopes which are not lattice pyramids, we can obtain all Gorenstein polytopes.
In each dimension, there exists only finitely many Gorenstein polytopes up to unimodular equivalence (\cite{Lag}),
and they are known up to dimension $4$ (\cite{Kre}).
The works \cite{BJ,HNT} also provide some classification results for Gorenstein polytopes in the high dimensional setting.

\smallskip
In this paper, we discuss a characterization of Gorenstein simplices in terms of their associated finite abelian groups.
In Section 1, we recall the Hermite normal form matrices and some of their properties that we will use in this paper.
In Section 2, we prove that a family of simplices arising from Hermite normal form matrices are Gorenstein (Theorem \ref{one}).
Using this result, we characterize Gorenstein simplices whose normalized volume is a prime number.
In fact, we will prove the following.
\begin{Theorem}
		\label{p}
	Let $p$ be a prime number and $\Delta \subset \RR^d$ a $d$-dimensional lattice simplex with normalized volume  $p$.
	Suppose that $\Delta$ is not a lattice pyramid over any lower-dimensional simplex.
	Then $\Delta$ is Gorenstein of index $r$ if and only if $d=rp-1$ and $\Lambda_{\Delta}$ is generated by $\left(\dfrac{1}{p},\ldots,\dfrac{1}{p}\right)$.
\end{Theorem} 
In Section 3, we extend these results by characterizing Gorenstein simplices whose normalized volume equals $p^2$ and $pq$, where $p$ and $q$ are prime numbers with $p \neq q$. In fact, we will prove the following theorems.
\begin{Theorem}
	\label{pp}
	Let $p$ be a prime number and $\Delta \subset \RR^d$ a $d$-dimensional lattice simplex with normalized volume $p^2$.
	Suppose that $\Delta$ is not a lattice pyramid over any lower-dimensional lattice simplex.
	Then $\Delta$ is Gorenstein of index $r$ if and only if one of the followings is satisfied:
	\begin{enumerate}
		\item There exists an integer $s$ with $0 \leq s \leq d-1$ such that $rp^2-1=(d-s)+ps$ and $\Lambda_{\Delta}$ is generated by $\left(\underbrace{\dfrac{1}{p},\ldots,\dfrac{1}{p}}_{s},\underbrace{\dfrac{1}{p^2},\ldots,\dfrac{1}{p^2}}_{d-s+1}\right)$ 
		for some ordering of the vertices of $\Delta$.
		\item    $d=rp-1$ and there exist an integer $s$ with $1 \leq s \leq d-1$    
		and integers $1 \leq a_1,\ldots,a_{s-1} \leq p-1$
		such that  $\Lambda_{\Delta}$ is generated by 
		$$\left(\dfrac{2-\sum\limits_{1 \leq i \leq s-1}a_i}{p},\dfrac{a_1+1}{p},\ldots,\dfrac{a_{s-1}+1}{p},0,\dfrac{1}{p},\ldots,\dfrac{1}{p}\right)$$ and $$\left(\dfrac{\left(\sum\limits_{1 \leq i \leq s-1}a_i\right)-1}{p},\dfrac{p-a_1}{p},\ldots,\dfrac{p-a_{s-1}}{p},\dfrac{1}{p},0,\ldots,0\right)$$
		for some ordering of the vertices of $\Delta$.
	\end{enumerate}
\end{Theorem}
\begin{Theorem}
		\label{pq}
	Let $p$ and $q$ be prime numbers with $p \neq q$ and 
	$\Delta \subset \RR^d$ a $d$-dimensional lattice simplex with normalized volume  $pq$.
	Suppose that $\Delta$ is not a lattice pyramid over any lower-dimensional lattice simplex.
	Then $\Delta$ is Gorenstein of index $r$ if and only if 
	there exist nonnegative integers $s_1,s_2,s_3$ with $s_1+s_2+s_3=d+1$ such that the following conditions are satisfied:
	\begin{enumerate}
		\item $rpq=s_1q+s_2p+s_3$;
		\item$\Lambda_{\Delta}$ is generated by
		$\left(
		\underbrace{\dfrac{1}{p},\ldots,\dfrac{1}{p}}_{s_1},\underbrace{\dfrac{1}{q},\ldots,\dfrac{1}{q}}_{s_2},\underbrace{\dfrac{1}{pq},\ldots,\dfrac{1}{pq}}_{s_3}\right)$ 	for some ordering of the vertices of $\Delta$.
	\end{enumerate}
\end{Theorem}
Moreover, we give a class of Gorenstein simplices whose normalized volume equals a power of a prime number (Theorem \ref{power}). 
Finally, in Section 4, we compute the volume of the associated dual reflexive simplices of the Gorenstein simplices described in Sections 2 and 3.

\begin{acknowledgement}{\rm
	The author would like to thank anonymous referees for reading the manuscript carefully.
	The author is partially supported by Grant-in-Aid for JSPS Fellows 16J01549.
}
\end{acknowledgement}

	\section{Preliminaries}
	In this section, we recall some basic facts about Hermite normal form matrices. 
For positive integers $d$ and $m$, we denote by $\text{Herm}(d,m)$ the finite set of lower triangular matrices $H=(h_{ij})_{1 \leq i,j \leq d} \in \ZZ_{\geq 0}^{d \times d}$ with determinant $m$ satisfying $h_{ij}<h_{ii}$ 
for all $i > j$.
It is well known that for any $M \in \ZZ^{d \times d}$ with determinant $m \in \ZZ_{>0}$ there exists a unimodular matrix $U \in \ZZ^{d \times d}$ and a \textit{Hermite normal form matrix} $H \in \text{Herm}(d,m)$ such that
$MU=H$.
Let  $\Delta \subset \RR^d$ be a lattice simplex of dimension $d$ with normalized volume $m$ and $v_0,\ldots,v_d$ the vertices of $\Delta$, and let $V$ be the $(d \times d)$-matrix whose $i$th row is $v_i-v_0$.
Then  one has $|\det(V)|=m$ and we may assume that $\det(V)=m$.
Hence there exist a unimodular matrix $U \in \ZZ^{d \times d}$ and a Hermite normal form matrix $H \in \text{Herm}(d,m)$ such that
$VU=H$.
In particular,  $\Delta$ is unimodularly equivalent to the lattice simplex whose vertices are the origin of $\RR^d$
and all rows of $H$.

Let $H=(h_{ij})_{1 \leq i,j \leq d} \in \ZZ_{\geq 0}^{d \times d}$ be a Hermite normal form matrix
and
set
$\ell(H)=\sharp\{i \mid h_{ii}>1\}$.
We then say that  $H$ has \textit{$\ell(H)$ nonstandard rows}.
Let $\Delta(H)$ be the lattice simplex whose vertices are the origin of $\RR^d$ and all rows of $H$,
and set $s=\max\{i \mid h_{ii}>1\}$.
If $\Delta(H)$ is not a lattice pyramid over any lower-dimensional lattice simplex, then $s=d$.
In \cite{HHL}, lattice simplices arising from Hermite normal form matrices are discussed.
\smallskip
	
We now recall a pair of lemmas that we will use in this paper.
	\begin{Lemma}[{\cite[Corollary 35.6]{HibiRedBook}}]
		\label{facet}
		Let $\Pc \subset \RR^d$ be a lattice polytope of
		dimension $d$ containing the origin in its interior. 
		Then a point $a \in \RR^d$ is a vertex of $\Pc^\vee$ if 
		and only if $\Hc \cap \Pc$ is a facet of $\Pc$,
		where $\Hc$ is the hyperplane
		$$\left\{ x \in \RR^d \ | \ \langle a, x \rangle =1 \right\}$$
		in $\RR^d$.
	\end{Lemma}
	By using this lemma, in order to see whether a lattice polytope is reflexive, we should compute the equations of supporting hyperplanes of facets of the polytope. 
	
	\begin{Lemma}[{\cite[Lemma 12]{Nill}}]
	\label{Nill}
	Let $\Delta \subset \RR^d$ be a lattice simplex of dimension $d$.
	Then $\Delta$ is a lattice pyramid if and only if there is $i \in \{0,\ldots,d\}$ such that $\lambda_i=0$ for all $(\lambda_0,\ldots,\lambda_d) \in \Lambda_\Delta$.
\end{Lemma}

	\section{Hermite normal form matrices with one nonstandard row}
	For a sequence of integers $A=(a_1,\ldots,a_{d-1},a_d)$ with $1 \leq a_1,\ldots, a_{d-1} \leq a_d$,
	we set $\Delta(A)=\text{conv}(v_0,\ldots,v_d) \subset \RR^d$, where 
	\begin{displaymath}
	v_i=\left\{
	\begin{aligned}
	&(0,\ldots ,0),& \ \textnormal{if}& \ i=0,\\
	&e_i,&\ \textnormal{if} &\  1 \leq i \leq d-1, \\
	&\sum\limits_{j=1}^{d-1} (a_d-a_j)e_j+a_de_d, &\ \textnormal{if} &\  i=d.
	\end{aligned}
	\right.
	\end{displaymath} 
	Here $e_{1}, \ldots, e_{d}$ are the canonical unit coordinate vectors of $\RR^{d}$.  
	Namely, $\Delta(A)$ is a lattice simplex arising from a Hermite normal form matrix with one nonstandard row.
	In particular, the lattice simplices $\Delta(A)$ are exactly the lattice simplices with one unimodular facet.
	
	At first, we give the equations of supporting hyperplanes of facets of $\Delta(A)$.
	
	\begin{Lemma}
		\label{fac1}
	For $0 \leq i \leq d$, let $\Fc_i$ be the facet of $\Delta(A)$ whose vertices are $v_0,\ldots,v_{i-1},v_{i+1},\ldots,v_d$ 
	and $\Hc_i$ the supporting hyperplane of $\Fc_i$.
	Then one has
	\begin{itemize}
		\item $\Hc_0=\{(x_1,\ldots,x_d) \in \RR^d :a_d\sum\limits_{j=1}^{d-1}x_j+(1-\sum\limits_{j=1}^{d-1}(a_d-a_j))x_d=a_d \}$;
		\item $\Hc_i=\{(x_1,\ldots,x_d) \in \RR^d :-a_dx_i+(a_d-a_i)x_d=0 \}$, $1 \leq i \leq d-1$;
			\item $\Hc_d=\{(x_1,\ldots,x_d) \in \RR^d : -x_{d}=0 \}$.
	\end{itemize}
	\end{Lemma}
	
	It is easy to compute $\Lambda_{\Delta(A)}$ for the simplex $\Delta(A)$,
	as demonstrated with the following lemma.
	\begin{Lemma}
		\label{onelem}
		Let $a_0$ be an integer with $1 \leq a_0 \leq a_d$ such that $a_d \mid (a_0+\cdots+a_{d-1}+1)$.
		Then the finite abelian group $\Lambda_{\Delta(A)}$ is generated by 	$\left(\dfrac{a_0}{a_d},\dfrac{a_1}{a_d},\ldots,\dfrac{a_{d-1}}{a_d},\dfrac{1}{a_d}\right).$
		In particular, $\Delta(A)$ is not a lattice pyramid over any lower-dimensional lattice simplex if and only if $1 \leq a_0,a_1,\ldots,a_{d-1} < a_d$.
	\end{Lemma}
	\begin{proof}
		Set $$(\lambda_0,\ldots,\lambda_d)=\left(\dfrac{a_0}{a_d},\dfrac{a_1}{a_d},\ldots,\dfrac{a_{d-1}}{a_d},\dfrac{1}{a_d}\right)　 \in (\RR/\ZZ)^{d+1}.$$
Then one has
		$$\sum\limits_{i=0}^{d}\lambda_i(v_i,1)=(b_1,\ldots,b_{d-1},1,\dfrac{a_0+\cdots+a_{d-1}+1}{a_d}) \in \ZZ^{d+1},$$
		where for $1 \leq i \leq d-1$, $b_i=\min\{1,a_d-a_i \}$.
Hence we know that $(\lambda_0,\ldots,\lambda_d)$ is an element of $\Lambda_{\Delta(A)}$.
Since the normalized volume of $\Delta(A)$ is $a_d$ and the order of $(\lambda_0,\ldots,\lambda_d)$ is $a_d$, $\Lambda_{\Delta(A)}$ is generated by  $(\lambda_0,\ldots,\lambda_d)$.
Moreover, by Lemma \ref{Nill},
it is follows that $\Delta(A)$ is not a lattice pyramid over any lower-dimensional lattice simplex if and only if $1 \leq a_0,a_1,\ldots,a_{d-1} < a_d$.
	\end{proof}

The following theorem characterizes exactly when the simplices $\Delta(A)$ are Gorenstein.
\begin{Theorem}
	\label{one}
	Suppose that $1 \leq a_0,\ldots,a_{d-1} <a_d$.
	Then $\Delta(A)$ is Gorenstein of index $r$ if and only if the following conditions are satisfied:
	\begin{itemize}
		\item For $0 \leq i \leq d-1$, $a_i \mid a_d$;
		\item $ra_d=a_0+\cdots+a_{d-1}+1$.
	\end{itemize}	
	\end{Theorem}

In order to prove this theorem, we show the following lemma. 

\begin{Lemma}
	\label{ver}
		Suppose that $1 \leq a_0,\ldots,a_{d-1} <a_d$,
				 $ra_d=a_0+\cdots+a_{d-1}+1$ and 
		for $0 \leq i \leq d-1$, $a_i \mid a_d$.
		Then
$\Delta(A)$ is Gorenstein of index $r$.
	Moreover, the vertices of the associated dual reflexive simplex are the following lattice points:
		 		 	 \begin{itemize}
		 \item $-e_d$;
		 \item $-\dfrac{a_d}{a_i}e_i+\dfrac{a_d-a_i}{a_i}e_d$ for  $1 \leq i \leq d-1$; 
		 \item $\dfrac{a_d}{a_0}\sum\limits_{j=1}^{d-1}e_j+\dfrac{(r-d+1)a_d-a_0}{a_0}e_d$.
		 \end{itemize} 
\end{Lemma}
\begin{proof}
	Since $a_d(d-1)+(1-\sum_{j=1}^{d-1}(a_d-a_j))=ra_d-a_1 < ra_d$,
	by Lemma \ref{fac1}, we know that $t=(1,\ldots,1)$ is an interior lattice point of $r\Delta(A)$.
	Set $\Delta=r\Delta(A)-t$.
	Then by Lemma \ref{fac1}, the equations of supporting hyperplanes of facets of $\Delta$ are as follows:
	\begin{itemize}
		\item $-x_{d}=1$;
		\item $-a_dx_i+(a_d-a_i)x_d=a_i$, $1 \leq i \leq d-1;$
		\item $a_d\sum\limits_{j=1}^{d-1}x_j+(1-\sum\limits_{j=1}^{d-1}(a_d-a_j))x_d=a_0$.
	\end{itemize}
Hence by Lemma \ref{facet}, $\Delta$ is reflexive and we can obtain the vertices of $\Delta^{\vee}$.
	\end{proof}
Now, we prove Theorem \ref{one}.
\begin{proof}[Proof of Theorem \ref{one}]
	Let $t=(t_1,\ldots,t_d) \in \RR^d$ be the unique interior lattice point of $r\Delta(A)$ and $\Delta'=r\Delta(A)-t$.
	Then for each $i$, one has $t_i  \geq 1$.
	By Lemma \ref{fac1}, the equation $-x_{d}=t_d$ is a supporting hyperplane of a facet of $\Delta'$.
	Hence by Lemma \ref{facet}, $w_0=-e_d/t_d$ is a vertex of $(\Delta')^{\vee}$.
	Therefore, we obtain $t_d=1$. 
	If for some $i$, $t_i \geq 2$, then $(t_1,\ldots,t_{i-1},t_i-1,t_{i+1},\ldots,t_{d-1},1)$ is the interior lattice point of $r\Delta(A)$.
	Since $(t_1,\ldots,t_d)$ is the unique interior lattice point of $r\Delta(A)$, one has 
	$t_1=\cdots=t_{d-1}=1$.
Therefore, by Lemma \ref{facet},  the following points are the vertices of $(\Delta')^{\vee}$:
\begin{displaymath}
w_i=\left\{
\begin{aligned}
&-e_d,& \ \textnormal{if}& \ i=0,\\
&-\dfrac{a_d}{a_i}e_i+\dfrac{a_d-a_i}{a_i}e_d,&\ \textnormal{if} &\  1 \leq i \leq d-1,\\
&\dfrac{a_d}{a}\sum\limits_{j=1}^{d-1}e_j+\dfrac{1-\sum\limits_{j=1}^{d-1}(a_d-a_j)}{a}e_d, &\ \textnormal{if} &\  i=d,
\end{aligned}
\right.
\end{displaymath} 
where $a=ra_d-\sum_{j=1}^{d-1}{a_j-1}$.
	Since the origin of $\RR^d$ belongs to the interior of $\Delta'$,
   we obtain $a>0$.
   Moreover, $\Delta'$ is reflexive, by Lemma \ref{facet}, it is known that  $a$ divides $a_d$.
   Hence  one has $1 \leq a < a_d$.
	Therefore, since $a_d \mid (a+a_1+\cdots+a_{d-1}+1)$ and $1 \leq a_0 < a_d$, we obtain  $a=a_0$.
	By Lemma \ref{ver}, this completes the proof.
\end{proof}

	By Lemmas \ref{facet} and \ref{ver}, 
	we can prove Theorem \ref{one}.
%
\begin{Remark}
If $\Delta(A)$ is Gorenstein of index $1$,
then $\Delta(A)$ is unimodularly equivalent to a lattice polytope $\Delta=\textnormal{conv}(e_1,\ldots,e_d, -\sum_{i=1}^{d}a_{i-1}e_i)$.
In \cite{BDS}, properties of this polytope $\Delta$ are discussed.
\end{Remark}
We obtain Theorem \ref{p} as a special case of Theorem \ref{one}.
\begin{proof}[Proof of Theorem \ref{p}]
	Since the normalized volume of $\Delta$ is a prime number, there exists a sequence of integers $A=(a_1,\ldots,a_{d-1},p)$ with $1 \leq a_1,\ldots,a_{d-1} \leq p$ such that $\Delta$ is unimodularly equivalent to $\Delta(A)$.
	Let $a_0$ be an integer with $1 \leq a_0 \leq p$ such that $p \mid (a_0+\cdots+a_{d-1}+1)$.
	Since $\Delta$ is not a lattice pyramid over any lower-dimensional simplex, by Lemma \ref{onelem}, 
	one has $1 \leq a_0,\ldots,a_{d-1}<p$.
	Hence, by Theorem \ref{one}, $\Delta(A)$ is Gorenstein of index $r$ if and only if
	$a_0=\cdots=a_{d-1}=1$ and $d=rp-1$.
	Therefore,  $\Delta$ is Gorenstein of index $r$ if and only if $d=rp-1$ and $\Lambda_{\Delta}$ is generated by $\left(\dfrac{1}{p},\ldots,\dfrac{1}{p}\right)$.
\end{proof}

\section{The case when $\text{Vol}(\Delta)=p^2$ or $\text{Vol}(\Delta)=pq$}
Let $s,d$ be positive integers with $1 \leq s<d$, and let $A=(a_1,\ldots,a_s)$ and $B=(b_1,\ldots,b_d)$ be sequences of integers with $0 \leq a_1,\ldots,a_{s-1} < a_s$ and $0 \leq b_1,\ldots,b_{d-1} < b_{d}$.
Set  $\Delta(A,B)=\text{conv}(v_0,\ldots,v_d) \subset \RR^d$, where 
\begin{displaymath}
v_i=\left\{
\begin{aligned}
&(0,\ldots ,0),& \ \textnormal{if}& \ i=0,\\
&e_i,&\ \textnormal{if} &\  1 \leq i \leq s-1, \\
&\sum\limits_{j=1}^{s} a_je_j, &\ \textnormal{if} &\  i=s,\\
&e_i,&\ \textnormal{if} &\  s+1 \leq i \leq d-1, \\
&\sum\limits_{j=1}^{d} b_je_j, &\ \textnormal{if} &\  i=d.
\end{aligned}
\right.
\end{displaymath}
Then $\Delta(A,B)$ is a lattice simplex arising from a Hermite normal form matrix with two nonstandard rows.

 We give the equations of supporting hyperplanes of facets of $\Delta(A,B)$.
\begin{Lemma}
			\label{fac2}
			Assume that $b_s=0$.
	For $0 \leq i \leq d$, let $\Fc_i$ be the facet of $\Delta(A,B)$ whose vertices are $v_0,\ldots,v_{i-1},v_{i+1},\ldots,v_d$ 
	and $\Hc_i$ the supporting hyperplane of $\Fc_i$.
	Then one has
	\begin{itemize}
		\item $\Hc_0=\{(x_1,\ldots,x_d) \in \RR^d :a_sb_d\sum\limits_{\substack{1 \leq j \leq d-1 \\ j \neq s}}x_j+b_d(1-\sum\limits_{1 \leq j \leq s-1}a_j)x_s+a_s(1-\sum\limits_{\substack{1 \leq j \leq d-1 \\ j \neq s}}b_j)x_d=a_sb_d \}$;
		\item $\Hc_i=\{(x_1,\ldots,x_d) \in \RR^d :-a_sb_dx_i+a_ib_dx_s+a_sb_ix_d=0 \}$, $1 \leq i \leq s-1$;
		\item $\Hc_s=\{(x_1,\ldots,x_d) \in \RR^d : -x_{s}=0 \}$;
		\item $\Hc_i=\{(x_1,\ldots,x_d) \in \RR^d :-b_dx_i+b_ix_d=0 \}$, $s+1 \leq i \leq d-1$;
		\item $\Hc_d=\{(x_1,\ldots,x_d) \in \RR^d : -x_{d}=0 \}$.
	\end{itemize}
\end{Lemma}

Let $p,q$ be prime numbers with $p \neq q$.
In this section, we characterize Gorenstein simplices whose normalized volume equals $p^2$ and $pq$.
In particular, we prove Theorems \ref{pp} and \ref{pq}.

We prove the following lemma.
\begin{Lemma}
	\label{lempq}
Let $p$ and $q$ be prime numbers
	and set $a_s=p$ and $b_d=q$.
	Suppose that $\Delta(A,B)$ is Gorenstein of index $r$.
	Then we have $b_s=0$ or $b_s=q-1$.
	Moreover, if $b_s=q-1$, then there exists a sequence of integers $C=(c_1,\ldots,c_{d-1},pq)$ with $1 \leq c_1,\ldots,c_{d-1} \leq pq$ such that $\Delta(A,B)$ and $\Delta(C)$ are unimodularly equivalent.
\end{Lemma}
\begin{proof}
	The following two equations define supporting hyperplanes of two facets of $r\Delta(A,B)$:
	\begin{itemize}
		\item $-x_{d}=0$;
		\item $-qx_{s}+b_sx_d=0$.
	\end{itemize}
	Let $t=(t_1,\ldots,t_d) \in \RR^d$ be the unique interior lattice point of $r\Delta(A,B)$.
	Then $t_i \geq 1$ for each $i$.
	Set $\Delta=r\Delta(A,B)-t$.
	Then the followings are equations of supporting hyperplanes of facets of $\Delta$:
	\begin{itemize}
		\item $-x_{d}=t_d$;
		\item $-qx_{s}+b_sx_d=qt_s-b_st_d$.
	\end{itemize}
	By Lemma \ref{facet}, $-\dfrac{e_d}{t_d}$ and $\dfrac{-qe_s+b_se_d}{qt_s-b_st_d}$ are vertices of $\Delta^{\vee}$.
	Hence since $\Delta$ is reflexive, we know that $t_d=1$ and $\dfrac{q}{qt_s-b_s}$ is an integer.
	Therefore, we have $t_s=1$ and $b_s \in \{0,q-1\}$.
	
	 Suppose that $b_s=q-1$. 
	Then we know
	$$\left(\dfrac{\lambda_0}{pq}, \dfrac{pq-a_1-pb_1}{pq},\ldots,\dfrac{pq-a_{s-1}-pb_{s-1}}{pq},\dfrac{1}{pq},\dfrac{q-b_{s+1}}{q},\ldots,\dfrac{q-b_{d-1}}{q},\dfrac{1}{q} \right)$$
	is an element of $\Lambda_{\Delta(A,B)}$,
	where $\lambda_0$ is an integer with $0 \leq \lambda_0 \leq pq-1$ such that the sum of all entries of this element is an integer.  
	Hence by Lemma \ref{onelem}, 
	there exists a sequence of integers $C=(c_1,\ldots,c_{d-1},pq)$ with $1 \leq c_1,\ldots,c_{d-1} \leq pq$ such that $\Delta(A,B)$ and $\Delta(C)$ are unimodularly equivalent.
\end{proof}

At first, we characterize Gorenstein simplices with normalized volume $p^2$.
In order to prove Theorem \ref{pp},  we show the following lemma.
\begin{Lemma}
	\label{ppver}
	Let $p$ be a prime number and set $a_s=b_d=p$
	and  
	$b_s=0$.
	Suppose that $d=rp-1$ and for $1 \leq i \leq s-1$, $a_i+b_i=p-1$ and for $s+1 \leq i \leq d-1$,  $b_i=p-1$. 
	Then $\Delta(A,B)$ is Gorenstein of index $r$.
	Moreover, the vertices of the associated dual reflexive simplex are the following lattice points:
	\begin{itemize}
		\item $-e_d$;
		\item $-pe_i+a_ie_s+b_ie_d$ for $1 \leq i \leq s-1$;
		\item $-e_s$;
		\item $-pe_i+b_ie_d$ for $s+1 \leq i \leq d-1$;
		\item $p\sum\limits_{\substack{1 \leq j \leq d-1 \\ j \neq s}}e_j+(1-\sum\limits_{1 \leq j \leq s-1}a_j)e_s+(1-\sum\limits_{\substack{1 \leq j \leq d-1 \\ j \neq s}}b_j)e_d$.
	\end{itemize}
\end{Lemma}
\begin{proof}
	Since $$p(d-2)+(1-\sum\limits_{1 \leq j \leq s-1}a_j)+(1-\sum\limits_{\substack{1 \leq j \leq d-1 \\ j \neq s}}b_j)=d=rp-1 <rp,$$
	by Lemma \ref{fac1}, we know that $t=(1,\ldots,1)$ is an interior lattice point of $r\Delta(A,B)$.
	Set $\Delta=r\Delta(A,B)-t$.
	Then by Lemma \ref{fac1}, 
	the equations of supporting hyperplanes of facets of $\Delta$ are as follows:
	\begin{itemize}
		\item $-x_{d}=1$;
		\item $-px_i+a_ix_s+b_ix_d=1$, $1 \leq i \leq s-1;$
		\item $-x_{s}=1$;
		\item $-px_i+b_ix_d=1$, $s+1 \leq i \leq d-1;$
		\item$p\sum\limits_{\substack{1 \leq j \leq d-1 \\ j \neq s}}x_j+(1-\sum\limits_{1 \leq j \leq s-1}a_j)x_s+(1-\sum\limits_{\substack{1 \leq j \leq d-1 \\ j \neq s}}b_j)x_d=1$.
	\end{itemize}
Hence by Lemma \ref{facet}, $\Delta$ is reflexive and we can obtain the vertices of $\Delta^{\vee}$.
\end{proof}

Now, we prove Theorem \ref{pp}.
	\begin{proof}[Proof of Theorem \ref{pp}]
		First notice that, by Theorem \ref{one}, the case of Hermite normal form matrices with one nonstandard row are captured in the statement (1).
		Hence,  we consider the case of Hermite normal form matrices with two nonstandard rows.
		Let $s,d$ be positive integers with $s<d$,
		and let $A=(a_1,\ldots,a_{s-1},p)$ and $B=(b_1,\ldots,b_{d-1},p)$ be sequences of integers with $0 \leq a_1,\ldots,a_{s-1} , b_1,\ldots,b_{d-1} < p$.
		Assume that $\Delta(A,B)$ is not a lattice pyramid over any lower-dimensional lattice simplex and $\Delta(A,B)$ is Gorenstein of index $r$.
		Then for $1 \leq i \leq s-1$, we have $(a_i,b_i)\neq (0,0)$ and for $s+1 \leq i \leq d-1$, we have $b_i\neq 0$.
		By Lemma \ref{lempq}, we only need to consider the case where $b_s=0$.
			If for some $1 \leq i \leq s-1$, $a_{i}=0$, then $\Delta(A,B)$ is unimodularly equivalent to $\Delta(A',B')$, where $A'=(a_1,\ldots,a_{i-1}, a_{i+1},\ldots,a_{s-1},p)$ and $B'=(b_1,\ldots,b_{i-1},b_{i+1},\ldots,b_{s-1},0,b_{i},b_{s+1},\ldots,b_{d-1},p)$.
		Hence we may assume that $a_1,\ldots,a_{s-1} \geq 1$.
		Let $t=(t_1,\ldots,t_d) \in \RR^d$ be the unique interior lattice point of $r\Delta(A,B)$,
		and set $\Delta'=r\Delta(A,B)-t$.
		Then by Lemma \ref{fac2}, the equations of supporting hyperplanes of facets of $\Delta'$ are as follows:
		\begin{itemize}
			\item $-x_{d}=t_d$;
			\item $-px_i+a_ix_s+b_ix_d=pt_i-a_it_s-b_it_d$, $1 \leq i \leq s-1;$
			\item $-x_{s}=t_s$;
			\item $-px_i+b_ix_d=pt_i-b_it_d$, $s+1 \leq i \leq d-1;$
			\item$p\sum\limits_{\substack{1 \leq j \leq d-1 \\ j \neq s}}x_j+(1-\sum\limits_{1 \leq j \leq s-1}a_j)x_s+(1-\sum\limits_{\substack{1 \leq j \leq d-1 \\ j \neq s}}b_j)x_d\\=rp-p\sum\limits_{\substack{1 \leq j \leq d-1 \\ j \neq s}}t_j-(1-\sum\limits_{1 \leq j \leq s-1}a_j)t_s-(1-\sum\limits_{\substack{1 \leq j \leq d-1 \\ j \neq s}}b_j)t_d$.
		\end{itemize}
		Hence by Lemma \ref{facet}, it is known that $-e_d/t_d$ and $-e_s/t_s$ are vertices of $(\Delta')^{\vee}$.
		Therefore, since $\Delta'$ is reflexive, we obtain $t_s=t_d=1$.
		Similarly, since $pt_i-a_i-b_i >0$ and $pt_i-a_i-b_i$ divides $p,a_i$ and $b_i$, and since $(a_i,b_i) \neq (0,0)$, we have that $pt_i-a_i-b_i=1$.
		Hence, for any $1 \leq i \leq s-1$, we have $t_i=1$ and $p-a_i-b_i=1$.
		Moreover, since  $b_i \neq 0$ for any $s+1 \leq i \leq d-1$, we have that $t_i=1$ and $p-b_i=1$.
		We then obtain
		$$rp-p\sum\limits_{\substack{1 \leq j \leq d-1 \\ j \neq s}}t_j-(1-\sum\limits_{1 \leq j \leq s-1}a_j)-(1-\sum\limits_{\substack{1 \leq j \leq d-1 \\ j \neq s}}b_j)
		=rp-d.$$
		Since $rp-d >0$ and $rp-d$ divides $p$, we have $rp-d=1$ or $rp-d=p$.
		
		
		Assume that $rp-d=p$.
		Then  since $p \mid (1-\sum_{1 \leq j \leq s-1}a_j)$,
		we know that $\Lambda_{\Delta(A,B)}$ is generated by 
		$$\left(0,\dfrac{a_1+1}{p},\ldots,\dfrac{a_{s-1}+1}{p},0,\dfrac{1}{p},\ldots,\dfrac{1}{p}\right)$$ and $$\left(0,\dfrac{p-a_1}{p},\ldots,\dfrac{p-a_{s-1}}{p},\dfrac{1}{p},0,\ldots,0\right).$$
		Therefore, by Lemma \ref{Nill}, $\Delta(A,B)$ is a lattice pyramid over a lower-dimensional lattice simplex.
		Thus one has $rp-d=1$.
		 Then it follows that 
		 $\Lambda_{\Delta(A,B)}$ is generated by 
		 $$\left(\dfrac{2-\sum\limits_{1 \leq i \leq s-1}a_i}{p},\dfrac{a_1+1}{p},\ldots,\dfrac{a_{s-1}+1}{p},0,\dfrac{1}{p},\ldots,\dfrac{1}{p}\right)$$ and $$\left(\dfrac{\left(\sum\limits_{1 \leq i \leq s-1}a_i\right)-1}{p},\dfrac{p-a_1}{p},\ldots,\dfrac{p-a_{s-1}}{p},\dfrac{1}{p},0,\ldots,0\right).$$
		 By Lemma \ref{ppver}, this completes the proof.
	\end{proof}
	
	Next, we characterize Gorenstein simplices with normalized volume $pq$.
In order to prove Theorem \ref{pq}, we show the following lemma.
\begin{Lemma}
	\label{pqver}
Let $p$ and $q$ be prime numbers with $p \neq q$
	and set $a_s=p$ and $b_d=q$.
	Assume that $k=rpq-p(d-s)-qs \in \{p,q\}$.
	Then $\Delta(A,B)$ is Gorenstein of index $r$. 
	Moreover, the vertices of the associated dual reflexive simplex are the following lattice points:
	 \begin{itemize}
	\item $-e_d$;
	\item $-pe_i+(p-1)e_{s}$ for $1 \leq i \leq s-1$;
	\item $-e_{s}$;  
	\item $-qe_i+(q-1)e_d$ for  $s+1 \leq i \leq d-1$; 
	\item $(c_1,\ldots,c_d)$,
\end{itemize}
where
\begin{displaymath}
c_i=\left\{
\begin{aligned}
&\dfrac{q(1-(s-1)(p-1))}{k},& \ \textnormal{if}& \ i=s,\\
&\dfrac{p(1-(d-s-1)(q-1)}{k},& \ \textnormal{if}& \ i=d,\\
&\dfrac{pq}{k},& \textnormal{o}& \textnormal{therwise}.\\
\end{aligned}
\right.
\end{displaymath} 
\end{Lemma}
\begin{proof}
Since
	$pq(d-2)+q(1-(p-1)(s-1))+p(1-(q-1)(d-s-1))=p(d-s)+qs < rpq$,
	by Lemma \ref{fac2}, it follows that $t=(1,\ldots,1) \in \ZZ^d$ is an interior lattice point of $r\Delta(A,B)$.
	Hence by Lemma \ref{fac2}, the equations of supporting hyperplanes of facets of $\Delta'=r\Delta(A,B)-t$ are as follows:
		\begin{itemize}
		\item $-x_{d}=1$;
		\item $-pqx_i+(p-1)qx_s=1$, $1 \leq i \leq s-1;$
		\item $-x_{s}=1$;
		\item $-qx_i+(q-1)x_d=1$, $s+1 \leq i \leq d-1;$
		\item$pq\sum\limits_{\substack{1 \leq j \leq d-1 \\ j \neq s}}x_j+q(1-(p-1)(s-1))x_s+p(1-(q-1)(d-s-1)x_d\\=rpq-p(d-s)-qs$.
	\end{itemize}
	    If $rpq-p(d-s)-qs=p$, then $p \mid s$. Hence, $p \mid (1-(p-1)(s-1))$.
	Moreover,  if $rpq-p(d-s)-qs=q$, then $q \mid (d-s)$, and so $q \mid (1-(d-s-1)(q-1)$.
	Thus by Lemma \ref{facet}, $\Delta'$ is reflexive and we can obtain the vertices of $(\Delta')^{\vee}$.
	\end{proof}
Now, we prove Theorem \ref{pq}.
\begin{proof}[Proof of Theorem \ref{pq}]
		The case when $s_3 \geq 1$ follows from Theorem \ref{one} since this case corresponds to the Hermite normal form matrices with one nonstandard row.
	Hence,  we consider the case of Hermite normal form matrices with two nonstandard rows.
	Let $s,d$ be positive integers with $s<d$ and $p,q$  prime numbers with $p \neq q$,
	and  let $A=(a_1,\ldots,a_{s-1},p)$ and $B=(b_1,\ldots,b_{d-1},q)$ be sequences of integers with $0 \leq a_1,\ldots,a_{s-1} < p$ and $0 \leq b_1,\ldots,b_{d-1} < q$.
	Assume that $\Delta(A,B)$ is not a lattice pyramid over any lower-dimensional lattice simplex and $\Delta(A,B)$ is Gorenstein  of index $r$.
	Then  for $1 \leq i \leq s-1$, we have $(a_i,b_i)\neq (0,0)$ and for $s+1 \leq i \leq d-1$, we have $b_i\neq 0$.
	By Lemma \ref{lempq}, we need only consider the case where $b_s=0$.
	
			Let $t=(t_1,\ldots,t_d) \in \RR^d$ be the unique interior lattice point of $r\Delta(A,B)$.
			Analogous to the proof in Theorem \ref{pp}, we have $t_i=1$ for each $i$
			and so we set $\Delta'=r\Delta(A,B)-t$.
			Then by Lemma \ref{fac2}, the equations of supporting hyperplanes of facets of $\Delta'$ are as follows:
			\begin{itemize}
				\item $-x_{d}=1$;
				\item $-pqx_i+a_iqx_s+pb_ix_d=pq-pb_i-a_iq$, $1 \leq i \leq s-1;$
				\item $-x_{s}=1$;
				\item $-qx_i+b_ix_d=q-b_i$, $s+1 \leq i \leq d-1;$
				\item$pq\sum\limits_{\substack{1 \leq j \leq d-1 \\ j \neq s}}x_j+q(1-\sum\limits_{1 \leq j \leq s-1}a_j)x_s+p(1-\sum\limits_{\substack{1 \leq j \leq d-1 \\ j \neq s}}b_j)x_d\\=rpq-pq(d-2)-q(1-\sum\limits_{1 \leq j \leq s-1}a_j)-p(1-\sum\limits_{\substack{1 \leq j \leq d-1 \\ j \neq s}}b_j)$.
			\end{itemize}
			Since $\Delta'$ is reflexive, by Lemma \ref{facet}, for $1 \leq i \leq s-1$ we have $pq-pb_i-a_iq \in \{1,p,q\}$ and for $s+1 \leq i \leq d-1$ we have $b_i=q-1$.
	       If for some $1 \leq i \leq s-1$, $pq-pb_{i}-a_{i}q=1$,
	       then since 
	       $$a=\left(\dfrac{\left(\sum\limits_{j=1}^{s-1}b_j\right)-(d-s)}{q},\dfrac{q-b_1}{q},\ldots,\dfrac{q-b_{s-1}}{q},0,\dfrac{1}{q},\ldots,\dfrac{1}{q} \right)$$
	    and 
	     $$b=\left(\dfrac{\left(\sum\limits_{j=1}^{s-1}a_j\right)-1}{p},\dfrac{p-a_1}{p},\ldots,\dfrac{p-a_{s-1}}{p},\dfrac{1}{p},0,\ldots,0 \right)$$
	     are elements of $\Lambda_{\Delta(A,B)}$,
	     we know that the $i$th entry of $a+b$ equals $\dfrac{1}{pq}$. Hence this is the case where $s_3 \geq 1$.
	     If for some $1 \leq i \leq s-1$, $pq-pb_{i}-a_{i}q=p$, then since $(a_{i},b_{i})=(0,q-1)$,
	     it follows that $\Delta(A,B)$ is unimodularly equivalent to $\Delta(A',B')$, where $A'=(a_1,\ldots,a_{i-1},a_{i+1}\ldots,a_{s-1},p)$ and $B'=(b_1,\ldots,b_{i-1},b_{i+1},\ldots,b_{s-1},0,b_{i},b_{s+1},\ldots,b_{d-1},q)$.
	      Hence we may assume that for any $1 \leq i \leq s-1$, we have that $pq-pb_i-a_iq=q$. In particular, $(a_i,b_i)=(p-1,0)$.
	     Then we know that 
	     an element $\left(
	     -\dfrac{p(d-s)+qs}{pq},\underbrace{\dfrac{1}{p},\ldots,\dfrac{1}{p}}_{s},\underbrace{\dfrac{1}{q},\ldots,\dfrac{1}{q}}_{d-s} \right)$ of $(\RR/\ZZ)^{d+1}$ generates  $\Lambda_{\Delta}$.
	     Moreover, we obtain 
	     $$1-\sum\limits_{1 \leq j \leq s-1}a_j=-p(s-1)+s,$$
       	$$1-\sum\limits_{\substack{1 \leq j \leq d-1 \\ j \neq s}}b_j=-q(d-s-1)+(d-s),$$
       	and
	     $$rpq-pq(d-2)-q(1-\sum\limits_{1 \leq j \leq s-1}a_j)-p(1-\sum\limits_{\substack{1 \leq j \leq d-1 \\ j \neq s}}b_j)=rpq-p(d-s)-qs.$$
	     Since $\Delta'$ is reflexive, by Lemma \ref{facet}, it follows that $rpq-p(d-s)-qs \in \{1,p,q,pq\}$.     
	     By Lemma \ref{Nill}, we know that $rpq-p(d-s)-qs \neq pq$. 
	      If $ rpq-p(d-s)-qs=1$, we have $\dfrac{-s}{p}+\dfrac{-d+s}{q}=\dfrac{-rpq+1}{pq}$.
	     Hence, this is again the case where $s_3 \geq 1$.
	     Therefore, we may just consider the case where $rpq-p(d-s)-qs \in \{p,q\}$. However, it is clear that this case satisfies the statement (2).
	     By Lemma \ref{pqver}, this completes the proof.
	   \end{proof}

   By Theorem \ref{one}, we can construct Gorenstein simplices whose normalized volume is equal $p^\ell$, where $p$ is a prime number and $\ell$ is a positive integer.
   Finally, we give other examples of Gorenstein simplices whose normalized volume equals $p^\ell$.
    These simplices arise from Hermite normal form matrices with $\ell$ nonstandard rows.  
   In particular, Theorem \ref{pp} $(2)$ is the motivation for the following theorem.
\begin{Theorem}
	\label{power}
Let $p$ be a prime number, and let $d$ and $\ell$ be positive integers with $\ell \leq d$, and let $1 \leq s_1< s_2 < \cdots < s_{\ell} = d$ be positive integers.
For $1 \leq i \leq k$ and $0 \leq j \leq d$,
we set 	
\begin{displaymath}
\RR/\ZZ \ni g_{ij}=\left\{
\begin{aligned}
&-\sum_{k=1}^{d}g_{ik},& \ \textnormal{if}& \ j=0,\\
&\dfrac{p-a_{ij}}{p},&\ \textnormal{if} &\  1 \leq j \leq s_i-1\ \textnormal{and\ } j \neq s_1,\ldots,s_{i-1}, \\
&\dfrac{1}{p}, &\ \textnormal{if} &\  j=s_i,\\
&0, &\ \textnormal{ot} &\textnormal{herwise},
\end{aligned}
\right.
\end{displaymath} 
where each $a_{ij}$ is  a positive integer with $1 \leq a_{ij} \leq p-1$.
Suppose that there exists  an integer $r$ with $d=rp-1$, and for $1 \leq j \leq d-1$ with $j \neq s_1,\ldots,s_{\ell}$, there exists a positive integer $t_j$ such that $\sum_{i}a_{ij}=t_jp-1$.
If $\Delta \subset \RR^d$ is a $d$-dimensional simplex such that $\Lambda_\Delta$ is generated by $(g_{10},\ldots,g_{1d}),\ldots,(g_{\ell 0},\ldots,g_{\ell d})$, 
then $\Delta$ is Gorenstein of index $r$ and $\textnormal{Vol}(\Delta)=p^\ell$.
\end{Theorem}
\begin{proof}
	Set $\Delta= \text{conv}(v_0,\ldots,v_d ) \subset \RR^d$, where
	\begin{displaymath}
	v_i=\left\{
	\begin{aligned}
	&(0,\ldots,0),& \ \textnormal{if}& \ i=0,\\
	&e_i,&\ \textnormal{if} & \  i \neq0, s_1,\ldots,s_{\ell}, \\
	&\sum_{\substack{1 \leq j < s_k \\ j \neq s_1,\ldots,s_{k-1}}}a_{ij}e_j+pe_{s_k}, &\ \textnormal{if} &\  i=s_k.\\
	\end{aligned}
	\right.
	\end{displaymath}
	Then $\Delta \subset \RR^d$ is a $d$-dimensional simplex such that $\text{Vol}(\Delta)=p^{\ell}$ and  $\Lambda_\Delta$ is generated by $(g_{10},\ldots,g_{1d}),\ldots,(g_{\ell 0},\ldots,g_{\ell d})$.
	Let $s_0=0$.
		Then the equations of supporting hyperplanes of facets of $r\Delta$ are as follows:
		\begin{itemize}
			\item $-x_{s_k}=0$, for $k=1,\ldots,\ell$;
			\item $-px_i+\sum_{j=k+1}^{\ell}a_{ji}x_{s_j}=0$, for $s_k < i < s_{k+1}$;
			\item $p\sum\limits_{\substack{1 \leq j < s_{\ell} \\ j \neq s_1,\ldots,s_{\ell-1}}}x_j+\sum\limits_{1\leq k \leq \ell}
			\left( \left(1-\sum\limits_{\substack{1 \leq j < s_k \\ j \neq s_1,\ldots,s_{k-1}}}a_{kj}\right)x_{s_k} \right)=rp$.
		\end{itemize}
		Let $t'=(t'_1,\ldots,t'_d)$ be a lattice point of $\RR^d$, where
			\begin{displaymath}
			t'_i=\left\{
			\begin{aligned}
			&1,& \ \textnormal{if}& \ i=s_1,\ldots,s_{\ell},\\
			&t_i,&\ \textnormal{if} & \  i \neq s_1,\ldots,s_{\ell}.\\
			\end{aligned}
			\right.
			\end{displaymath}
			Now, we claim $t'$ is an interior lattice point of $r\Delta$.
			Indeed,
			for $s_k < i < s_{k+1}$, we have
			$$-pt_i+\sum\limits_{j=k+1}^{\ell}a_{ji}=-1<0$$
			and
			$$p\sum\limits_{\substack{1 \leq j < s_{\ell} \\ j \neq s_1,\ldots,s_{\ell-1}}}t_j+\sum\limits_{1\leq k \leq \ell}
			\left( 1-\sum\limits_{\substack{1 \leq j < s_k \\ j \neq s_1,\ldots,s_{k-1}}}a_{kj} \right)=d=rp-1<rp.$$
		
			Now set $\Delta'=r\Delta-b$.
				Then the equations of supporting hyperplanes of facets of $\Delta'$ are as follows:
				\begin{itemize}
					\item $-x_{s_k}=1$, for $k=1,\ldots,\ell$;
					\item $-px_i+\sum\limits_{j=k+1}^{\ell}a_{ji}x_{s_j}=1$, for $s_k < i < s_{k+1}$;
					\item $p\sum\limits_{\substack{1 \leq j < s_{\ell} \\ j \neq s_1,\ldots,s_{\ell-1}}}x_j+\sum\limits_{1\leq k \leq \ell}
					\left( \left(1-\sum\limits_{\substack{1 \leq j < s_k \\ j \neq s_1,\ldots,s_{k-1}}}a_{kj}\right)x_{s_k} \right)=1$.
				\end{itemize}
				Hence by Lemma \ref{facet}, $\Delta'$ is reflexive, and so $\Delta$ is Gorenstein of index $r$.
\end{proof}

\begin{Remark}
	\label{rmpower}
Let $\Delta$ be the Gorenstein simplex as in Theorem \ref{power}.
Then the vertices of the associated dual reflexive simplex of $\Delta$ are following lattice points:
	\begin{itemize}
		\item $-e_{s_k}$, for $k=1,\ldots,\ell$;
		\item $-pe_i+\sum\limits_{j=k+1}^{\ell}a_{ji}e_{s_j}$, for $s_k < i < s_{k+1}$;
		\item $p\sum\limits_{\substack{1 \leq j < s_{\ell} \\ j \neq s_1,\ldots,s_{\ell-1}}}e_j+\sum\limits_{1\leq k \leq \ell}
		\left( \left(1-\sum\limits_{\substack{1 \leq j < s_k \\ j \neq s_1,\ldots,s_{k-1}}}a_{kj}\right)e_{s_k} \right)$.
	\end{itemize}
\end{Remark}

\section{volume of the associated dual reflexive simplex}
In this section, we compute the volume of the associated dual reflexive simplices of the Gorenstein simplices we constructed in Sections 2 and 3.
 We first consider the case of Gorenstein simplices arising from Hermite normal form matrices with one nonstandard row.
\begin{Theorem}
	\label{dual}
	Let $\Delta(A) \subset \RR^d$ be a $d$-dimensional Gorenstein simplex of index $r$ as in Theorem \ref{one} and set $\Delta=r\Delta(A)-(1,\ldots,1)$.
	For $0 \leq i \leq d-1$, we set $b_i=a_d/a_i$.
	Then  we have $\textnormal{Vol}(\Delta^\vee)=r\prod_{j=0}^{d-1}b_i$. 
\end{Theorem}
\begin{proof}
	 By Lemma \ref{ver}, we know that $\Delta^\vee=\text{conv}(w_0,\ldots,w_d)$, where
	 \begin{displaymath}
	 w_i=\left\{
	 \begin{aligned}
	 &-e_d,& \ \textnormal{if}& \ i=0,\\
	 &-\dfrac{a_d}{a_i}e_i+\dfrac{a_d-a_i}{a_i}e_d,&\ \textnormal{if} &\  1 \leq i \leq d-1, \\
	 &\dfrac{a_d}{a_0}\sum\limits_{j=1}^{d-1}e_j+\dfrac{(r-d+1)a_d-a_0}{a_0}e_d, &\ \textnormal{if} &\  i=d.
	 \end{aligned}
	 \right.
	 \end{displaymath} 
	 It is easy show $\Delta^\vee$ is unimodularly equivalent to a $d$-dimensional simplex $\Delta'$ whose vertices $v_0',\ldots,v_d'$ are the following:
	  \begin{displaymath}
	  v_i'=\left\{
	  \begin{aligned}
	  &(0,\ldots,0),& \ \textnormal{if}& \ i=0,\\
	  &-b_ie_i,&\ \textnormal{if} &\  1 \leq i \leq d-1, \\
	  &b_0\sum\limits_{j=1}^{d-1}e_j+rb_0e_d,&\ \textnormal{if} &\  i = d. \\
	  \end{aligned}
	  \right.
	  \end{displaymath} 
	  Hence we have $\textnormal{Vol}(\Delta^\vee)=r\prod_{j=0}^{d-1}b_i$, as desired.
	\end{proof}
	From this theorem, we immediately obtain the following corollary.
	\begin{Corollary}
		\label{pv}
		Let $\Delta \subset \RR^d$ be a $d$-dimensional Gorenstein simplex of index $r$ whose normalized volume equals a prime number $p$.
		Suppose that $\Delta$ is not a lattice pyramid over any lower-dimensional lattice simplex and the unique interior lattice point of $r\Delta$ is the origin of $\RR^d$.
		Then we have $\textnormal{Vol}((r\Delta)^\vee)=rp^d$.
	\end{Corollary}

Next, we consider the case of Gorenstein simplices with normalized volume $p^2$, where $p$ is a prime number.
By Theorem \ref{dual}, we can compute the volume of the associated dual reflexive simplices of the Gorenstein simplices in Theorem \ref{pp} (1).

The Gorenstein simplices in Theorem \ref{pp} (2) are included in the Gorenstein simplices in Theorem \ref{power}.
Hence, we consider the case of the Gorenstein simplices in Theorem \ref{power}.
In fact, we can obtain the following Theorem.

\begin{Theorem}
	Let $\Delta \subset \RR^d$ be a $d$-dimensional Gorenstein polytope of index $r$ as in Theorem \ref{power}
	such that the unique interior lattice point of $r\Delta$ is the origin in $\RR^d$.
	Then we have $\textnormal{Vol}((r\Delta)^{\vee})=rp^{d-\ell+1}$.
\end{Theorem}

\begin{proof}
	By Remark \ref{rmpower}, 
	$(r\Delta)^\vee$ is the convex hull of the following lattice points:
	\begin{itemize}
		\item $-e_{s_k}$, for $k=1,\ldots,\ell$;
		\item $-pe_i+\sum\limits_{j=k+1}^{\ell}a_{ji}e_{s_j}$, for $s_k < i < s_{k+1}$;
		\item $p\sum\limits_{\substack{1 \leq j < s_{\ell} \\ j \neq s_1,\ldots,s_{\ell-1}}}e_j+\sum\limits_{1\leq k \leq \ell}
		\left( \left(1-\sum\limits_{\substack{1 \leq j < s_k \\ j \neq s_1,\ldots,s_{k-1}}}a_{kj}\right)e_{s_k} \right)$.
	\end{itemize}
	So we set
	\begin{displaymath}
	U=\left(
	\begin{array}{ccccc}
	1 & 0 & \cdots & 0& t'_1\\
	\vdots & \vdots & & \vdots&\vdots\\
	0 & 0 & \cdots & 1 & t'_{d-1}\\
	0 & 0 & \cdots & 0& t'_d
	\end{array}
	\right) \in \ZZ^{d \times d}.
	\end{displaymath} 
	Since $t'_d=t_{\ell}=1$, it follows that $U$ is unimodular.
	Letting $\Delta'=f_U((r\Delta)^\vee+e_d)$,
	we know that $\Delta'$ is the convex hull of the following lattice points:
	\begin{itemize}
		\item $(0,\ldots,0)$;
		\item $-e_{s_k}$, for $k=1,\ldots,\ell-1$;
		\item $-pe_i+\sum\limits_{j=k+1}^{\ell-1}a_{ji}e_{s_j}$, for $s_k < i < s_{k+1}$;
		\item $p\sum\limits_{\substack{1 \leq j < s_{\ell} \\ j \neq s_1,\ldots,s_{\ell-1}}}e_j+\sum\limits_{1\leq k \leq \ell-1}
		\left( \left(1-\sum\limits_{\substack{1 \leq j < s_k \\ j \neq s_1,\ldots,s_{k-1}}}a_{kj}\right)e_{s_k} \right)+rpe_d$.
	\end{itemize}
	Hence we have $\textnormal{Vol}(\Delta')=rp^{d-\ell+1}$, as desired.
\end{proof}
\begin{Corollary}
		Let $p$ be a prime number, and 
	let $\Delta \subset \RR^d$ be a $d$-dimensional Gorenstein simplex of index $r$ whose normalized volume equals $p^2$.\\
	(1) Suppose that $\Delta$ and $s$ satisfy the condition of Theorem \ref{pp} (1) and the unique interior lattice point of $r\Delta$ is the origin in $\RR^d$.
	Then  we have $\textnormal{Vol}((r\Delta)^\vee)=rp^{2d-s}$.\\
	(2) Suppose that  $\Delta$ satisfies the condition of Theorem \ref{pp} (2) and the unique interior lattice point of $r\Delta$ is the origin in $\RR^d$.
		Then  we have $\textnormal{Vol}((r\Delta)^\vee)=rp^{d-1}$.
\end{Corollary}

Finally, we consider the case of Gorenstein simplices whose normalized volume equals $pq$, where $p$ and $q$ are prime numbers with $p \neq q$.
\begin{Theorem}
	\label{pqvol}
	Let $p$ and $q$ be prime integers with $p \neq q$ and 
	$\Delta \subset \RR^d$ a $d$-dimensional Gorenstein simplex of index $r$ whose normalized volume equals $pq$.
	Suppose that $\Delta$ is not a lattice pyramid over any lower-dimensional lattice simplex and the unique interior lattice point of $r\Delta$ is the origin in $\RR^d$.
	Then we have $\textnormal{Vol}((r\Delta)^\vee)=
	rp^{s_1+s_3-1}q^{s_2+s_3-1}$, where
		$s_1,s_2,s_3$ are nonnegative integers which satisfy the conditions of Theorem \ref{pq}.
\end{Theorem}

\begin{proof}
	First, assume that $s_3 \geq 1$.
	Then by Theorem \ref{dual}, we obtain $\textnormal{Vol}((r\Delta)^\vee)=
	rp^{s_1+s_3-1}q^{s_2+s_3-1}$.
	
	Next,  assume that $s_3=0$.
		Then by the condition (1) of Theorem \ref{pq}, we know that $(s_1,s_2) \neq (1,d)$ and $(s_1,s_2) \neq (d,1)$.
		Moreover, by the condition (2) of Theorem \ref{pq} and the normalized volume of $\Delta$, we have $(s_1,s_2) \neq (d+1,0)$ and $(s_1,s_2) \neq (0,d+1)$.
	Hence, we have $s_1,s_2 \geq 2$. 
	Since $\Lambda_{\Delta}$ is generated by 
		$\left(
		\underbrace{\dfrac{1}{p},\ldots,\dfrac{1}{p}}_{s_1},\underbrace{\dfrac{1}{q},\ldots,\dfrac{1}{q}}_{s_2}\right)$,
		we may assume that 
		$r\Delta=r\Delta(A,B)-(1,\ldots,1)$,
		where $A=(\underbrace{p-1,\ldots,p-1}_{s_1-1},p)$ and $B=(\underbrace{0,\ldots,0}_{s_1},\underbrace{q-1,\ldots,q-1}_{s_2-2},q)$.
		Then by Lemma \ref{pqver}, we know that $(r\Delta)^\vee=\text{conv}(w_0,\ldots,w_d)$ where
				 \begin{displaymath}
				 w_i=\left\{
				 \begin{aligned}
				 &-e_d,& \ \textnormal{if}& \ i=0,\\
				 &-pe_i+(p-1)e_{s_1},&\ \textnormal{if} &\  1 \leq i \leq s_1-1, \\
				 &-e_i,& \ \textnormal{if}& \ i=s_1,\\
				 &-qe_i+(q-1)e_d,&\ \textnormal{if} &\  s_1+1 \leq i \leq d-1, \\
				 &(c_1,\ldots,c_d), &\ \textnormal{if} &\  i=d,
				 \end{aligned}
				 \right.
				 \end{displaymath} 
				 and
			 \begin{displaymath}
				 c_i=\left\{
				 \begin{aligned}
				 &\dfrac{q(1-(s_1-1)(p-1))}{p},& \ \textnormal{if}& \ i=s_1,\\
				 &1-(d-s_1-1)(q-1),& \ \textnormal{if}& \ i=d,\\
				 &q,& \textnormal{o}& \textnormal{therwise}.\\
				 \end{aligned}
				 \right.
				 \end{displaymath} 
			  It is easy show $(r\Delta)^\vee$ is unimodularly equivalent to a $d$-dimensional simplex $\Delta'$ whose vertices $v_0',\ldots,v_d'$ are the following:
			  \begin{displaymath}
			  v_i'=\left\{
			  \begin{aligned}
			  &(0,\ldots,0),& \ \textnormal{if}& \ i=0,\\
			  &-(p-1)e_i+pe_{s_1},& \ \textnormal{if}& \ 1 \leq i \leq s_1-1,\\
			  &-e_i,& \ \textnormal{if}& \ i=s_1,\\
			  &-qe_i,& \ \textnormal{if}& \ s_1+1 \leq i \leq d-1,\\
			  &(c_{1},\ldots,c_{d-1},c_1+\cdots +c_d+1),&\ \textnormal{if} &\  i=d, \\
			  \end{aligned}
			  \right.
			  \end{displaymath} 
			  Since $c_1+\cdots+c_d+1=rq$,
			     we have that $\textnormal{Vol}((r\Delta)^\vee)=rp^{s_1-1}q^{s_2-1}$,
			     as desired.
	\end{proof}

		\end{document}